\documentclass[twoside,10pt]{article}%
\usepackage{amssymb}
\usepackage{amsfonts}
\usepackage{amsmath}
\usepackage{graphicx}%
\providecommand{\U}[1]{\protect\rule{.1in}{.1in}}
 \arraycolsep=1.5pt 

\newtheorem{theorem}{Theorem}

\newtheorem{definition}[theorem]{Definition}

\newtheorem{lemma}[theorem]{Lemma}

\newtheorem{proposition}[theorem]{Proposition}
\newtheorem{remark}[theorem]{Remark}

\newenvironment{proof}[1][Proof]{\noindent\textbf{#1.} }{\ \rule{0.5em}{0.5em}}
\begin{document}

\title{\textbf{Blowup of Regular} \textbf{Solutions for the Relativistic
Euler-Poisson Equations}}
\author{\textsc{Wai Hong Chan\thanks{E-mail address: waihchan@ied.edu.hk}, Sen
Wong\thanks{E-mail address: senwongsenwong@yahoo.com.hk} and }M\textsc{anwai
Yuen\thanks{Corresponding Author, E-mail address: nevetsyuen@hotmail.com }}\\\textit{Department of Mathematics and Information Technology,}\\\textit{The Hong Kong Institute of Education,}\\\textit{10 Lo Ping Road, Tai Po, New Territories, Hong Kong}}
\date{Revised 05-Sept-2015}
\maketitle

\begin{abstract}
In this paper, we study the blowup phenomena for the regular solutions of the
isentropic relativistic Euler-Poisson equations with a vacuum state in
spherical symmetry. Using a general family of testing functions, we obtain new
blowup conditions for the relativistic Euler-Poisson equations. We also show
that the proposed blowup conditions are valid regardless of the speed
requirement, which was one of the key constraints stated in "Y. Geng,
\textit{Singularity Formation for Relativistic Euler and Euler-Poisson
Equations with Repulsive Force}, Commun. Pure Appl. Anal., \textbf{14} (2015), 549--564.".


\ 


\end{abstract}

\section{Introduction}

The isentropic relativistic Euler-Poisson equations \cite{1} are expressed as
follows:%
\begin{equation}
\left\{
\begin{matrix}
\displaystyle\partial_{t}\left(  \frac{n}{\sqrt{1-|v|^{2}/c^{2}}}\right)
+{\nabla}\cdot\left(  \frac{nv}{\sqrt{1-|v|^{2}/c^{2}}}\right)  =0\text{,}\\
\\
\displaystyle\partial_{t}\left(  \frac{p/c^{2}+\rho}{1-|v|^{2}/c^{2}}v\right)
+{\nabla}\cdot\left(  \frac{p/c^{2}+\rho}{1-|v|^{2}/c^{2}}v\otimes v\right)
+{\nabla}p=\frac{4\pi n{\nabla}\phi}{\sqrt{1-|v|^{2}/c^{2}}}\text{,}\\
\\
\displaystyle\Delta\phi=\frac{4\pi n}{\sqrt{1-|v|^{2}/c^{2}}}\text{,}%
\end{matrix}
\right.  \label{e1}%
\end{equation}
where $n$ is defined by%
\begin{equation}
\rho=n\left(  1+e/c^{2}\right)  \text{,} \label{ee2}%
\end{equation}
which satisfies%
\begin{equation}
\frac{\mathrm{d}n}{nc^{2}}=\frac{\mathrm{d}\rho}{p+\rho c^{2}}\text{.}
\label{ee3}%
\end{equation}
The unknowns and constants in the above equation are defined as
follows:\newline$\rho:[0,\infty)\times\mathbf{R}^{3}\rightarrow\lbrack
0,\infty)$ and $n:[0,\infty)\times\mathbf{R}^{3}\rightarrow\lbrack0,\infty)$
denote the proper mass-energy density and the charge density, respectively;
$c$ is the speed of light; $v:[0,\infty)\times\mathbf{R}^{3}\rightarrow
\mathbf{R}^{3}$ is the velocity of an electro-fluid; $-\phi:[0,\infty
)\times\mathbf{R}^{3}\rightarrow\mathbf{R}$ is the electrostatic potential in
the inertial frame; and $p=p(\rho)$ is the pressure function of the
electro-fluid in a proper frame. The equation of state $p$ follows the
$\gamma$-law:%
\begin{equation}
p=\rho^{\gamma}\text{,} \label{ee4}%
\end{equation}
where $\gamma>1$ is the adiabatic index. Lastly, the constant $e\geq0$ in
$(\ref{ee2})$ is the specific internal energy.

Relativistic electrodynamics includes the study of the interaction between
relativistic charged particles and electromagnetic fields when the particles
are moving at a speed comparable to the speed of light in a vacuum. At such a
high speed, the motion of the charged particles no longer obeys the Newtonian
equations, so relativistic equations of particles must be applied. Under a
field with a much stronger electric than magnetic effect, such as those in
supernova explosions, gravitational collapse, and the formation and expansion
of black holes and neutron stars, the motion of an isentropic relativistic
electro-fluid can be described by the Euler-Poisson equations $(\ref{ee2})$
when the charged particles are moving very fast.

To understand the mathematical nature of relativistic fluid dynamics, we first
review a related and previously developed relativistic model, namely, the
relativistic Euler equations. Makino and Ukai $\cite{9,10}$ and Lefloch and
Ukai $\cite{11}$ established the local existence of classical solutions to the
relativistic system using the theory of a quasi-linear symmetric hyperbolic
system. More precisely, the critical part of Makino and Ukai's proof was based
on the existence of a strictly convex entropy for the non-vacuum case; the
critical part of Lefloch and Ukai's proof relied on the generalized Riemann
invariants and normalized velocity for the vacuum case. Geng and Li
$\cite{12}$ extended these results to the isentropic system. For the
non-isentropic system, Guo $\cite{15}$ proved the blowup result for smooth
solutions using the averaged quantities method developed by Sideris
$\cite{13,14}$. Moreover, Pan and Smoller $\cite{16}$ applied the classical
energy method to show the singularity formation of smooth solutions.

Due to the complexity of the structures of system $(\ref{e1})$, research on
multi-dimensional relativistic Euler-Poisson equations is still at an early
stage. In 2013, Mai, Li and Zhang \cite{2} gave the first well-posed result
for the steady-state relativistic Euler-Poisson equations with relaxation. For
system $(\ref{e1})$ in the one dimensional case, Geng and Wang $\cite{4}$
obtained the global existence of a smooth solution with some monotonic
conditions on the initial data. The importance of system $(\ref{e1})$ is that
the non-relativistic Euler-Poisson equations are the Newtonian limit of system
$(\ref{e1})$. Readers may refer to $\cite{Perthame,7,8,W}$ for the blowup
results of the non-relativistic Euler-Poisson equations.

In this paper, we consider the spherical symmetric solutions, namely,%
\begin{equation}
n=n(t,r)\text{, \ \ \ \ \ }\rho=\rho(t,r)\text{, \ \ \ \ \ }v=\frac{x}%
{r}v(t,r)\text{,}%
\end{equation}
where $r=|x|$ is the radius of the spatial variables $x\in\mathbf{R}^{3}%
$.\newline Then, system $(\ref{e1})$ is transformed into%
\begin{equation}
\left\{
\begin{matrix}
\displaystyle\partial_{t}\left(  \frac{n}{\sqrt{1-v^{2}/c^{2}}}\right)
+\partial_{r}\left(  \frac{n}{\sqrt{1-v^{2}/c^{2}}}v\right)  +\frac{2}{r}%
\frac{n}{\sqrt{1-v^{2}/c^{2}}}v=0\text{,}\\
\\
\displaystyle\partial_{t}\left(  \frac{p/c^{2}+\rho}{1-v^{2}/c^{2}}v\right)
+\partial_{r}\left(  \frac{p/c^{2}+\rho}{1-v^{2}/c^{2}}v^{2}\right)
+\partial_{r}p+\frac{2}{r}\frac{p/c^{2}+\rho}{1-v^{2}/c^{2}}v^{2}=\frac{4\pi
n\phi_{r}}{\sqrt{1-v^{2}/c^{2}}}\text{,}\\
\\
\displaystyle\phi_{rr}=-\frac{8\pi}{r^{2}}\int_{0}^{r}\frac{n}{\sqrt
{1-v^{2}/c^{2}}}s^{2}ds+\frac{4\pi n}{\sqrt{1-v^{2}/c^{2}}}\text{.}%
\end{matrix}
\right.  \label{e7}%
\end{equation}
Note that we have%
\begin{equation}
\phi_{r}=\frac{4\pi}{r^{2}}\int_{0}^{r}\frac{n}{\sqrt{1-v^{2}/c^{2}}}s^{2}ds
\label{ee8}%
\end{equation}
as equation $(\ref{e7})_{3}$ can be solved using Green's function.

In $\cite{7}$, Yuen obtained a blowup result for the compressible Euler and
Euler-Poisson equations with repulsive forces using the integration method.
Recently, Geng $\cite{1}$ used the integration method described in $\cite{7}$
to obtain a blowup result for the regular solutions of relativistic Euler and
Euler-Poisson equations. Note that in $\cite{5,6}$, the authors generalized
the testing function in the integration method to any strictly increasing
function. By combining the method in $\cite{5,6}$ with the results of Geng
$\cite{1}$, we obtain the following theorem, which is our main contribution.
Moreover, we remove the condition that $|v|\geq c/2$ in $\cite{1}$.

\begin{theorem}
\label{THM1}For any strictly increasing $C^{1}$ function $f(r)$ vanishing at
$r=0$, the regular solutions of system $(\ref{e7})$ with initial data
$(\ref{8})$ and%
\begin{equation}
{p}^{\prime}(\rho)<ac^{2}\text{, \ \ \ \ \ for some }a\in(0,1)
\end{equation}
blow up on or before the finite time%
\begin{equation}
T=\displaystyle\frac{2\left(  \int_{0}^{R}f(r)v_{0}dr\right)  \left(  \int
_{0}^{R}\frac{f^{2}(r)}{{f}^{\prime}(r)}dr\right)  }{\left(  \int_{0}%
^{R}f(r)v_{0}dr\right)  ^{2}-2C\left(  \int_{0}^{R}\frac{f^{2}(r)}{{f}%
^{\prime}(r)}dr\right)  \left(  \int_{0}^{R}f(r)dr\right)  }>0\text{,}%
\end{equation}
if%
\begin{equation}
\int_{0}^{R}f(r)v_{0}dr>\sqrt{2C\left(  \int_{0}^{R}\frac{f^{2}(r)}%
{{f}^{\prime}(r)}dr\right)  \left(  \int_{0}^{R}f(r)dr\right)  }\text{,}%
\end{equation}
where%
\begin{equation}
C=\frac{c^{2}(\gamma+a-\gamma a+9)}{2(\gamma-1)(1-a)^{2}}>0\text{.}%
\end{equation}

\end{theorem}

\section{Integration Method}

We consider the Cauchy problem of system $(\ref{e7})$ with the following
initial data:%
\begin{equation}
\left\{
\begin{matrix}
\rho(0,r)=\rho_{0}(r)\geq0\text{,}\\
v(0,r)=v_{0}(r)<c\text{,}\\
\text{supp}(\rho_{0},v_{0})\subseteq\{r:r\leq R\}\text{, \ \ \ \ \ for some
}R>0\text{.}%
\end{matrix}
\right.  \label{8}%
\end{equation}
Note that for system $(\ref{e7})$ to be well-defined, we must have%
\begin{equation}
|v|<c\text{,}%
\end{equation}
which is guaranteed by $(\ref{8})_{2}$.

First, we show that the sign of $\rho$ is determined by its initial value.

\begin{lemma}
\label{ll2}If $n_{0}$ is positive, then $n$ is always positive, where $n_{0}$
is defined by%
\begin{equation}
\rho_{0}=n_{0}\left(  1+e/c^{2}\right)  \text{.}%
\end{equation}

\end{lemma}

\begin{proof}
Let $r(t;r_{0})$ be a characteristic curve starting at $r_{0}$. That is,
$r(t;r_{0})$ satisfies the following ordinary differential equation:%
\begin{equation}
\left\{
\begin{matrix}
\frac{\mathrm{d}}{\mathrm{d}t}r(t;r_{0})=v(t,r(t;r_{0}))\text{,}\\
r(0;r_{0})=r_{0}\text{.}%
\end{matrix}
\right.  \label{e11}%
\end{equation}
Then, for any $C^{1}$ function $F(t,r)$, by the chain rule, we have%
\begin{equation}
\frac{\mathrm{d}}{\mathrm{d}t}F(t,r(t;r_{0}))=F_{t}(t,r(t;r_{0}%
))+v(t,r(t;r_{0}))F_{r}(t,r(t;r_{0}))\text{.} \label{e12}%
\end{equation}
Take%
\begin{equation}
F=\frac{n}{\sqrt{1-v^{2}/c^{2}}}\text{,}%
\end{equation}
then $(\ref{e7})_{1}$ becomes%
\begin{equation}
\frac{\mathrm{d}}{\mathrm{d}t}F+\left(  v_{r}+2v/r\right)  F=0\text{,}%
\end{equation}
where we omit the substitution $r=r(t;r_{0})$. The above ordinary differential
equation can be solved using an integral factor and the solution is%
\begin{equation}
F(t,r(t;r_{0}))=F(0,r_{0})\exp\left(  -\int_{0}^{t}\left(  v_{r}+2v/r\right)
(s,r(s;r_{0}))ds\right)  \text{.} \label{e15}%
\end{equation}
Thus,%
\begin{equation}
n_{0}>0\Rightarrow F(0,r_{0})>0\Rightarrow F(t,r(t;r_{0}))>0\Rightarrow
n(t,r(t;r_{0}))>0\text{.}%
\end{equation}
The proof is complete.
\end{proof}

\begin{remark}
By $(\ref{ee2})$, $\rho$ and $n$ have the same sign. Thus, it follows from the
initial data $(\ref{8})$ that $\rho$ is always non-negative.
\end{remark}

\begin{lemma}
\label{ll2}For any $C^{1}$ non-vacuum ($\rho\neq0$) solutions of system
$(\ref{e7})$ with the condition%
\begin{equation}
{p}^{\prime}(\rho)<c^{2}\text{,}%
\end{equation}
we have the following relation:%
\begin{equation}
\displaystyle v_{t}+\frac{1-{p}^{\prime}/c^{2}}{1-{p}^{\prime}v^{2}/c^{4}%
}vv_{r}+\frac{(1-v^{2}/c^{2})^{2}{p}^{\prime}}{q(1-{p}^{\prime}v^{2}/c^{4}%
)}\rho_{r}=\frac{4\pi n(\sqrt{1-v^{2}/c^{2}})^{3}}{q(1-{p}^{\prime}v^{2}%
/c^{4})}\phi_{r}+\frac{2(1-v^{2}/c^{2})v^{2}{p}^{\prime}/c^{2}}{r(1-{p}%
^{\prime}v^{2}/c^{4})}\text{,} \label{ee18}%
\end{equation}
where%
\begin{equation}
q:=p/c^{2}+\rho\label{ee19}%
\end{equation}
and%
\begin{equation}
{p}^{\prime}:={p}^{\prime}(\rho)\text{, \ \ \ \ is\ the derivative of }p\text{
with respect to }\rho\text{.}%
\end{equation}

\end{lemma}

\begin{proof}
From $(\ref{ee3})$, we have%
\begin{equation}
\frac{\mathrm{d}n}{n}=\frac{\mathrm{d}\rho}{q}\text{.} \label{e3}%
\end{equation}
From $(\ref{e7})_{1}$, we have%
\begin{equation}%
\begin{array}
[c]{c}%
\displaystyle\frac{n}{q\sqrt{1-v^{2}/c^{2}}}\rho_{t}+\frac{n}{c^{2}\left(
\sqrt{1-v^{2}/c^{2}}\right)  ^{3}}vv_{t}+\frac{n}{q\sqrt{1-v^{2}/c^{2}}}%
v\rho_{r}+\frac{n}{\left(  \sqrt{1-v^{2}/c^{2}}\right)  ^{3}}v_{r}\\
\displaystyle+\frac{2n}{r\sqrt{1-v^{2}/c^{2}}}v=0
\end{array}
\end{equation}
or%
\begin{equation}
\displaystyle\rho_{t}+v\rho_{r}=-\frac{q}{c^{2}(1-v^{2}/c^{2})}vv_{t}-\frac
{q}{1-v^{2}/c^{2}}v_{r}-\frac{2qv}{r}\text{. \ \ \ \ \ (For }n\neq0\text{ and
}q\neq0\text{)} \label{e23}%
\end{equation}

On the other hand, from $(\ref{e7})_{2}$, we have%
\begin{equation}
\displaystyle\frac{1+{p}^{\prime}/c^{2}}{1-v^{2}/c^{2}}(\rho_{t}+v\rho
_{r})v+\frac{q(1+v^{2}/c^{2})}{(1-v^{2}/c^{2})^{2}}v_{t}+\frac{2q}%
{(1-v^{2}/c^{2})^{2}}vv_{r}+{p}^{\prime}\rho_{r}+\frac{2q}{r(1-v^{2}/c^{2}%
)}v^{2}=\frac{4\pi n}{\sqrt{1-v^{2}/c^{2}}}\phi_{r}\text{.} \label{eq241}%
\end{equation}
Substituting $(\ref{e23})$ into $(\ref{eq241})$, we have%
\begin{equation}
\displaystyle\frac{(1-{p}^{\prime}v^{2}/c^{4})q}{(1-v^{2}/c^{2})^{2}}%
v_{t}+\frac{(1-{p}^{\prime}/c^{2})q}{(1-v^{2}/c^{2})^{2}}vv_{r}-\frac
{2q{p}^{\prime}/c^{2}}{r(1-v^{2}/c^{2})}v^{2}+{p}^{\prime}\rho_{r}=\frac{4\pi
n}{\sqrt{1-v^{2}/c^{2}}}\phi_{r}%
\end{equation}
or (for $q\neq0$),%
\begin{equation}
\displaystyle v_{t}+\frac{1-{p}^{\prime}/c^{2}}{1-{p}^{\prime}v^{2}/c^{4}%
}vv_{r}+\frac{(1-v^{2}/c^{2})^{2}{p}^{\prime}}{q(1-{p}^{\prime}v^{2}/c^{4}%
)}\rho_{r}=\frac{4\pi n(\sqrt{1-v^{2}/c^{2}})^{3}}{q(1-{p}^{\prime}v^{2}%
/c^{4})}\phi_{r}+\frac{2(1-v^{2}/c^{2})v^{2}{p}^{\prime}/c^{2}}{r(1-{p}%
^{\prime}v^{2}/c^{4})}\text{.}%
\end{equation}
The proof is complete.
\end{proof}

Following Makino, Ukai and Kawashima \cite{3}, Geng in \cite{1} introduced a
corresponding notion of regular solutions for the relativistic Euler-Poisson
system, defined as follows.

\begin{definition}
A solution $(\rho,v)$ of system $(\ref{e7})$ is regular if the following
conditions are satisfied:\newline(1)%
\begin{equation}
\displaystyle\left(  p^{\frac{\gamma-1}{2\gamma}},v\right)  \in C^{1}\text{.}%
\end{equation}
(2)%
\begin{equation}
|(v^{2})_{r}|\leq c^{2}\text{.} \label{ee49}%
\end{equation}
(3)%
\begin{equation}
|({p}^{\prime})_{r}|\leq c^{2}\text{.} \label{ee50}%
\end{equation}

\end{definition}

\begin{remark}
Condition ($1$) in the above definition implies that a regular solution is
$C^{1}$.
\end{remark}

In $\cite{1}$, Geng showed the non-expanding support of the regular solutions
of system $(\ref{e7})$. For the sake of completeness, we include a proof here.

\begin{proposition}
The supports for the regular solutions of system $(\ref{e7})$ with initial
conditions $(\ref{8})$ do not expand. That is to say,%
\[
\text{supp}(\rho,v)\subseteq\{r:r\leq R\}\text{ \ \ \ \ \ for all }t\text{.}%
\]

\end{proposition}

\begin{proof}
As in the proof of Lemma $\ref{ll2}$, let $r(t;r_{0})$ be a characteristic
curve starting at $r_{0}$. Then, when $r_{0}\geq R$, $\rho_{0}(r_{0}%
)=n_{0}(r_{0})=0$. From $(\ref{e15})$,%
\begin{equation}
n(t,r(t;r_{0}))=\rho(t,r(t;r_{0}))=0\text{.}%
\end{equation}
On the other hand, note that for $\rho\neq0$,%
\begin{equation}
\frac{{p}^{\prime}}{q}\rho_{r}=\frac{\gamma\rho^{\gamma-2}\rho_{r}}%
{\rho^{\gamma-1}/c^{2}+1}=\frac{2\gamma/(\gamma-1)}{w^{2}/c^{2}+1}ww_{r}%
\end{equation}
and%
\begin{equation}
\frac{n}{q}=\frac{1}{(1+e/c^{2})\left(  \rho^{\gamma-1}/c^{2}+1\right)
}\text{,}%
\end{equation}
where%
\begin{equation}
w:=p^{\frac{\gamma-1}{2\gamma}}%
\end{equation}
is $C^{1}$. Thus, by continuity, when $\rho=0$, $w=0$, and hence%
\begin{equation}
\frac{{p}^{\prime}}{q}\rho_{r}=0\text{.}%
\end{equation}
\newline Similarly, when $\rho=0$, by continuity we have%
\begin{equation}
\frac{n}{q}=\frac{1}{1+e/c^{2}}\text{.}%
\end{equation}
Thus, when $r_{0}\geq R$, after substitution, $(\ref{ee18})$ becomes%
\begin{equation}
v_{t}+vv_{r}=0\text{. \ \ \ \ \ \ (Note that }{p}^{\prime}=0\text{ when }%
\rho=0\text{.)} \label{ee36}%
\end{equation}
From $(\ref{e12})$ ($F=v$), we have%
\begin{equation}
\frac{\mathrm{d}}{\mathrm{d}t}v(t,r(t;r_{0}))=0
\end{equation}
or%
\begin{equation}
v(t,r(t;r_{0}))=v_{0}(r_{0})=0\text{.}%
\end{equation}
From $(\ref{e11})$, we have%
\begin{equation}
\displaystyle\frac{\mathrm{d}}{\mathrm{d}t}r(t;r_{0})=v(t,r(t;r_{0}))=0\text{,
\ \ \ \ \ when }r_{0}\geq R\text{.}%
\end{equation}
Thus,%
\begin{equation}
r(t;r_{0})=r(0;r_{0})=r_{0}\text{, \ \ \ \ \ when }r_{0}\geq R.
\end{equation}
Therefore, we have%
\begin{equation}
\rho(t,r_{0})=\rho(t,r(t;r_{0}))=0
\end{equation}
and%
\begin{equation}
v(t,r_{0})=v(t,r(t;r_{0}))=0
\end{equation}
when $r_{0}\geq R$.\newline As $r_{0}\geq R$ is arbitrary, the proof is complete.
\end{proof}

Now, we are ready to present the proof of the theorem.

\begin{proof}
[Proof of Thoerem $\ref{THM1}$]By Lemma $\ref{ll2}$ and $(\ref{ee8})$, we have%
\begin{equation}
\displaystyle v_{t}+\frac{1-{p}^{\prime}/c^{2}}{1-{p}^{\prime}v^{2}/c^{4}%
}vv_{r}+\frac{(1-v^{2}/c^{2})^{2}{p}^{\prime}}{q(1-{p}^{\prime}v^{2}/c^{4}%
)}\rho_{r}=\frac{4\pi n(\sqrt{1-v^{2}/c^{2}})^{3}}{q(1-{p}^{\prime}v^{2}%
/c^{4})}\phi_{r}+\frac{2(1-v^{2}/c^{2})v^{2}{p}^{\prime}/c^{2}}{r(1-{p}%
^{\prime}v^{2}/c^{4})}\geq0\text{.}%
\end{equation}
Thus,%
\begin{equation}
v_{t}+vv_{r}-\frac{{p}^{\prime}/c^{2}(1-v^{2}/c^{2})}{1-{p}^{\prime}%
v^{2}/c^{4}}vv_{r}+\frac{(1-v^{2}/c^{2})^{2}{p}^{\prime}}{q(1-{p}^{\prime
}v^{2}/c^{4})}\rho_{r}\geq0\text{.} \label{6}%
\end{equation}

Now multiply $f(r)$ and take integration over $[0,R]$ with respect to $r$ on
both sides of $(\ref{6})$ to obtain%
\begin{equation}
\int_{0}^{R}f(r)v_{t}dr+\int_{0}^{R}f(r)vv_{r}dr+\int_{0}^{R}f(r)\frac
{{p}^{\prime}/c^{2}(1-v^{2}/c^{2})}{1-{p}^{\prime}v^{2}/c^{4}}(-vv_{r}%
)dr+\int_{0}^{R}f(r)\frac{(1-v^{2}/c^{2})^{2}{p}^{\prime}}{q(1-{p}^{\prime
}v^{2}/c^{4})}\rho_{r}dr\geq0\text{,} \label{eee21}%
\end{equation}
where $R$ is given by the initial data $(\ref{8})$.

First, the integrand of the third term in expression $(\ref{eee21})$ is less
than or equal to%
\begin{equation}
f(r)\frac{{p}^{\prime}/c^{2}(1-v^{2}/c^{2})}{1-{p}^{\prime}v^{2}/c^{4}}%
|vv_{r}|\leq f(r)\frac{1\cdot1}{1-a}|vv_{r}|\leq f(r)\frac{c^{2}}%
{2(1-a)}\text{ \ \ \ \ \ (by }(\ref{ee49})\text{).}%
\end{equation}
Thus, expression $(\ref{eee21})$ becomes%
\begin{equation}
\int_{0}^{R}f(r)v_{t}dr+\int_{0}^{R}f(r)vv_{r}dr+\frac{c^{2}}{2(1-a)}\int
_{0}^{R}f(r)dr+\int_{0}^{R}f(r)\frac{(1-v^{2}/c^{2})^{2}{p}^{\prime}}%
{q(1-{p}^{\prime}v^{2}/c^{4})}\rho_{r}dr\geq0\text{.} \label{eee22}%
\end{equation}

Second, denote the fourth term in expression $(\ref{eee22})$ to be $I$.%
\begin{align}
I  &  =\int_{0}^{R}\frac{f(r)(1-v^{2}/c^{2})^{2}\gamma\rho^{\gamma-1}}{\left(
\rho^{\gamma}/c^{2}+\rho\right)  (1-{p}^{\prime}v^{2}/c^{4})}d\rho\text{
\ \ \ \ \ (by }(\ref{ee4})\text{ and }(\ref{ee19})\text{)}\\
&  =\int_{0}^{R}\frac{f(r)(1-v^{2}/c^{2})^{2}}{(1-{p}^{\prime}v^{2}/c^{4}%
)}\left(  \frac{\gamma\rho^{\gamma-1}}{\rho^{\gamma}/c^{2}+\rho}\right)
d\rho\\
&  =\frac{\gamma c^{2}}{\gamma-1}\int_{0}^{R}\frac{f(r)(1-v^{2}/c^{2})^{2}%
}{1-{p}^{\prime}v^{2}/c^{4}}d\left[  \ln\left(  1+\rho^{\gamma-1}%
/c^{2}\right)  \right] \\
&  =-\frac{\gamma c^{2}}{\gamma-1}\int_{0}^{R}\ln\left(  1+\rho^{\gamma
-1}/c^{2}\right)  d\left[  \frac{f(r)(1-v^{2}/c^{2})^{2}}{1-{p}^{\prime}%
v^{2}/c^{4}}\right]  \text{ \ \ \ \ \ (by integration by parts)}\\
&  =-\frac{\gamma c^{2}}{\gamma-1}\int_{0}^{R}\ln\left(  1+\rho^{\gamma
-1}/c^{2}\right)  \left[  \frac{(1-v^{2}/c^{2})^{2}}{1-{p}^{\prime}v^{2}%
/c^{4}}\right]  df(r)\\
&  -\frac{\gamma c^{2}}{\gamma-1}\int_{0}^{R}\ln\left(  1+\rho^{\gamma
-1}/c^{2}\right)  f(r)d\left[  \frac{(1-v^{2}/c^{2})^{2}}{1-{p}^{\prime}%
v^{2}/c^{4}}\right] \nonumber\\
&  \leq-\frac{\gamma c^{2}}{\gamma-1}\int_{0}^{R}\ln\left(  1+\rho^{\gamma
-1}/c^{2}\right)  f(r)d\left[  \frac{(1-v^{2}/c^{2})^{2}}{1-{p}^{\prime}%
v^{2}/c^{4}}\right] \\
&  =-\frac{\gamma c^{2}}{\gamma-1}\int_{0}^{R}\ln\left(  1+\rho^{\gamma
-1}/c^{2}\right)  f(r)\left(  \frac{2(1-v^{2}/c^{2})\left(  -(v^{2})_{r}%
/c^{2}\right)  }{1-{p}^{\prime}v^{2}/c^{4}}+(1-v^{2}/c^{2})^{2}\frac
{d({p}^{\prime}v^{2}/c^{4})}{\left(  1-{p}^{\prime}v^{2}/c^{4}\right)  ^{2}%
}\right)  dr\\
&  =-\frac{\gamma}{\gamma-1}\int_{0}^{R}\ln\left(  1+\rho^{\gamma-1}%
/c^{2}\right)  f(r)\left(  \frac{2(1-v^{2}/c^{2})\left(  -(v^{2})_{r}\right)
}{1-{p}^{\prime}v^{2}/c^{4}}+(1-v^{2}/c^{2})^{2}\frac{\left[  ({p}^{\prime
})_{r}v^{2}+({p}^{\prime})(v^{2})_{r}\right]  /c^{2}}{\left(  1-{p}^{\prime
}v^{2}/c^{4}\right)  ^{2}}\right)  dr\\
&  =-\frac{\gamma}{\gamma-1}\int_{0}^{R}f(r)\ln\left(  1+\rho^{\gamma-1}%
/c^{2}\right)  \frac{(1-v^{2}/c^{2})^{2}v^{2}/c^{2}}{(1-{p}^{\prime}%
v^{2}/c^{4})^{2}}({p}^{\prime})_{r}dr\\
&  -\frac{\gamma}{\gamma-1}\int_{0}^{R}f(r)\ln\left(  1+\rho^{\gamma-1}%
/c^{2}\right)  \frac{(1-v^{2}/c^{2})(-2+{p}^{\prime}v^{2}/c^{4}+{p}^{\prime
}/c^{2})}{(1-{p}^{\prime}v^{2}/c^{4})^{2}}(v^{2})_{r}dr\text{.}\nonumber
\end{align}
\newline By the elementary inequality%
\begin{equation}
\ln(1+x)\leq x\text{, \ \ \ \ for }x\geq0\text{,}%
\end{equation}
we have%
\begin{equation}
\ln(1+\rho^{\gamma-1}/c^{2})\leq\rho^{\gamma-1}/c^{2}=\frac{{p}^{\prime}%
}{\gamma c^{2}}<\frac{1}{\gamma}\text{.}%
\end{equation}
Thus,%
\begin{align}
I  &  \leq\frac{\gamma}{\gamma-1}\int_{0}^{R}f(r)\ln\left(  1+\rho^{\gamma
-1}/c^{2}\right)  \frac{(1-v^{2}/c^{2})^{2}v^{2}/c^{2}}{(1-{p}^{\prime}%
v^{2}/c^{4})^{2}}|({p}^{\prime})_{r}|dr\\
&  +\frac{\gamma}{\gamma-1}\int_{0}^{R}f(r)\ln\left(  1+\rho^{\gamma-1}%
/c^{2}\right)  \frac{(1-v^{2}/c^{2})(2+{p}^{\prime}v^{2}/c^{4}+{p}^{\prime
}/c^{2})}{(1-{p}^{\prime}v^{2}/c^{4})^{2}}|(v^{2})_{r}|dr\nonumber\\
&  \leq\frac{c^{2}}{(\gamma-1)(1-a)^{2}}\int_{0}^{R}f(r)dr+\frac{4c^{2}%
}{(\gamma-1)(1-a)^{2}}\int_{0}^{R}f(r)dr\text{ \ \ \ \ \ \ (by }%
(\ref{ee50})\text{ and }(\ref{ee49})\text{)}\\
&  =\frac{5c^{2}}{(\gamma-1)(1-a)^{2}}\int_{0}^{R}f(r)dr\text{.}%
\end{align}
It follows that expression $(\ref{eee22})$ becomes%
\begin{equation}
\int_{0}^{R}f(r)v_{t}dr+\int_{0}^{R}f(r)vv_{r}dr+C\int_{0}^{R}f(r)dr\geq
0\text{,}%
\end{equation}
where%
\begin{equation}
C:=\frac{c^{2}}{2(1-a)}+\frac{5c^{2}}{(\gamma-1)(1-a)^{2}}=\frac{c^{2}%
(\gamma+a-\gamma a+9)}{2(\gamma-1)(1-a)^{2}}>0\text{.}%
\end{equation}

Third, by integration by parts, we have%
\begin{equation}
\int_{0}^{R}f(r)v_{t}dr-\frac{1}{2}\int_{0}^{R}v^{2}{f}^{\prime}%
(r)dr+C\int_{0}^{R}f(r)dr\geq0\text{.} \label{24}%
\end{equation}
Let%
\begin{equation}
H(t):=\int_{0}^{R}f(r)v(t,r)dr\text{.}%
\end{equation}
Then,%
\begin{equation}
{H}^{\prime}(t)=\int_{0}^{R}f(r)v_{t}dr\text{,} \label{58}%
\end{equation}
and, by the integral version of the Cauchy inequality, we have%
\begin{equation}
H^{2}(t)=\left(  \int_{0}^{R}fvdr\right)  ^{2}=\left(  \int_{0}^{R}\frac
{f}{{f}^{\prime}}vdf\right)  ^{2}\leq\left(  \int_{0}^{R}\frac{f^{2}}%
{{f}^{\prime2}}df\right)  \left(  \int_{0}^{R}v^{2}df\right)
\end{equation}
or%
\begin{equation}
H^{2}(t)\leq\left(  \int_{0}^{R}\frac{f^{2}(r)}{{f}^{\prime}(r)}dr\right)
\left(  \int_{0}^{R}{f}^{\prime}(r)v^{2}dr\right)  \text{.} \label{59}%
\end{equation}
From $(\ref{58})$ and $(\ref{59})$, $(\ref{24})$ becomes%
\begin{equation}
\displaystyle{H}^{\prime}(t)-\frac{H^{2}(t)}{2\int_{0}^{R}\frac{f^{2}(r)}%
{{f}^{\prime}(r)}dr}+C\int_{0}^{R}f(r)dr\geq0
\end{equation}
or%
\begin{equation}
{H}^{\prime}(t)\geq\frac{H^{2}(t)}{2B_{1}}-B_{2}\text{,} \label{61}%
\end{equation}
where%
\begin{equation}
B_{1}:=\int_{0}^{R}\frac{f^{2}(r)}{{f}^{\prime}(r)}dr>0
\end{equation}
and%
\begin{equation}
B_{2}=C\int_{0}^{R}f(r)dr>0\text{.}%
\end{equation}
From $(\ref{61})$, it is well-known that if%
\begin{equation}
H(0)>\sqrt{2B_{1}B_{2}}\text{,} \label{eee28}%
\end{equation}
then the solutions blow up on or before the finite time%
\begin{equation}
T:=\frac{2B_{1}H(0)}{H^{2}(0)-2B_{1}B_{2}}>0. \label{eee29}%
\end{equation}
More precisely, from $(\ref{eee28})$ and the continuity of $H(t)$,
${H}^{\prime}(t)>0$ for all $t\geq0$. Moreover, from $(\ref{61})$,%
\begin{align}
{H}^{\prime}(t)  &  \geq\frac{H^{2}(t)}{2B_{1}}-B_{2}\\
&  =\frac{1}{2B_{1}}\left(  1-\frac{2B_{1}B_{2}}{H^{2}(0)}\right)
H^{2}(t)+B_{2}\left(  \frac{H^{2}(t)}{H^{2}(0)}-1\right) \\
&  \geq\frac{1}{2B_{1}}\left(  1-\frac{2B_{1}B_{2}}{H^{2}(0)}\right)
H^{2}(t)\\
&  =\frac{H^{2}(0)-2B_{1}B_{2}}{2B_{1}H^{2}(0)}H^{2}(t)\text{.}%
\end{align}
It follows that%
\begin{equation}
H(t)\geq\left(  \frac{1}{H(0)}-\frac{H^{2}(0)-2B_{1}B_{2}}{2B_{1}H^{2}%
(0)}t\right)  ^{-1}\text{.}%
\end{equation}
Thus, the blowup time $T$ in $(\ref{eee29})$ is obtained.

The proof is complete.
\end{proof}

\section{Acknowledgement}

This research was partially supported by the Dean's Research Grant FLASS/ECR-9
from the Hong Kong Institute of Education.

\end{document}